\documentclass[10pt]{article}

\usepackage{amsthm,amsmath,amssymb,natbib,setspace}

\newcommand{\Mod}{\mathbf{mod}}

\def\Var{{\rm Var}\,}
     \def\E{{\rm E}\,}

     \def\rng{\mathrm{Rng} }
     \def\dist{ \overset{\mathrm{iid}}{\sim}}
     \def\Exp{\mathrm{Exp}}

\newtheorem{theorem}{Theorem}
   \newtheorem{proposition}[theorem]{Proposition}
       \newtheorem{lemma}[theorem]{Lemma}

    \numberwithin{equation}{section}
    \numberwithin{theorem}{section}
    
    \theoremstyle{definition}
    \newtheorem{remark}[theorem]{Remark}
   \newtheorem{example}[theorem]{Example}

\begin{document}

\title{The Exponential Distribution Analog of the Grubbs--Weaver Method}

 \author{Andrew V. Sills\footnote{Andrew Sills was partially supported by Grant H98230-14-1-0159 
 from the National Security Agency when some of the work for this project was completed. 
 Sills (\texttt{asills@georgiasouthern.edu}) is the corresponding author.} \ and Charles W. Champ\\ Georgia Southern University\\ \small{Department of Mathematical
 Sciences, Statesboro, Georgia}}
\date{\today}
\maketitle

\begin{abstract}
 Grubbs and Weaver (1947) suggest a minimum-variance unbiased estimator for the population standard deviation of a normal random variable, where a random sample is drawn
and a weighted sum of the ranges of subsamples is calculated.  The optimal choice involves using as many subsamples of size eight as possible.  They verified their results numerically for samples of size up to 100, and conjectured that their ``rule of eights" is valid for all sample sizes.  Here we examine the analogous problem where the underlying distribution is exponential and find that a ``rule of fours" yields optimality and prove the result rigorously. 
\end{abstract}

\section{Introduction}

Suppose
\[  X_1, X_2, \dots, X_n \dist f(x) \]
is a random sample of size $n$ where each $X_i$ is a continuous random variable 
with density function $f(x)$, and (unknown) standard deviation $\sigma$. 
 Denote the order statistics as
\[ X_{1:n} \leq X_{2:n} \leq \cdots \leq X_{n:n}. \]
The range of the sample is
\[ R_n = X_{n:n} - X_{1:n}  \] and the standardized sample range is
\[ W_n = \frac{R_n}{\sigma}. \]

In the case where the sample is drawn from a normal population,
the minimum variance unbiased estimator of $\sigma$ is known to be the 
bias corrected sample
standard deviation $S/c$, with
\[ S = \sqrt{ \frac{\sum_{i=1}^n(X_i - \bar{X})^2}{n-1} }  \]
and
\[ c = \sqrt{\frac{2}{n-1}} \frac{\Gamma(\frac{n}{2})}{\Gamma(\frac{n-1}{2})}, \]
where \[ \Gamma(x)= \int_0^\infty t^{x-1} e^{-t}\ dt \] is Euler's gamma function. 
On the other hand, an easily calculuated unbiased estimator for $\sigma$ is
\[ \frac{R_n}{\E(W_n)} ,\] which of course has a comparatively large variance.

\citet{GW} study a compromise between $S/c$ and $R_n/E(W_n)$ in the case where
the random sample is drawn from a normal population with mean $\mu$ and 
(unknown) variance
$\sigma^2$.  They partition the sample of size $n$ into $m$
subsamples of sizes $n_1, n_2, \dots, n_m$ respectively, where each $n_i \geq 2$, and
$n_1 + n_2 + \dots + n_m = n$.  Then estimate $\sigma$ by $\hat{\sigma}$, where
\begin{equation} \label{sigmahatdef}
\hat{\sigma} = \sum_{i=1}^m a_i R_{n_i}
\end{equation}
with $R_{n_i}$ representing the range of the $i$th subsample, which is of size $n_i$
(see Remark~\ref{remark} below for clarification),
and $a_1, a_2, \dots, a_m$ are a set of weights chosen to guarantee that $\hat{\sigma}$
will be unbiased.  They use the term ``group range'' to mean the range of the 
subsample, and thus entitle their paper ``The best unbiased estimate of population 
standard deviation based on group ranges.''
  
Recall that a \emph{partition} $\lambda$ of a positive integer $n$ is a representation 
of $n$ as an unordered sum of positive integers. Each summand is called a \emph{part}
of the partition, and the number of parts in a given partition is called its \emph{length}.
Since the order of the parts is irrelevant, it is often convenient to write the
partition $\lambda$ of length $m$ as
$\lambda = (n_1, n_2, \dots, n_m)$ where 
$n_1 \geq n_2 \geq \dots \geq n_m$ and $n_1 + n_2 + \cdots + n_m = n$.  Thus, e.g., the five partitions
of $n=4$ are as follows:
\[ (4) \qquad (3,1) \qquad (2,2) \qquad (2,1,1) \qquad (1,1,1,1). \]

  Alternatively, we may define $f_i = f_i(\lambda)$ to denote the \emph{frequency}
(or \emph{multiplicity}) of part $i$
in the partition $\lambda$, i.e. the number of times the part $i$ appears, and employ the
``frequency superscript notation" $\langle 1^{f_1} 2^{f_2} 3^{f_3} \cdots \rangle$.
In this notation, the five partitions of $n=4$ are
\[ \langle 4 \rangle \qquad \langle 1\ 3\rangle \qquad \langle 2^2 \rangle \qquad
\langle 1^2\ 2\rangle \qquad \langle 1^4 \rangle, \]
 where we have followed the convention that $f_i = 1$ the superscript is omitted and
 $i^{f_i}$ is omitted if $f_i=0$.
 
  We are interested in partitions of $n$ where all
parts are at least $2$, as a subsample of size $1$ will, by definition, have a range of $0$.
Let us call such a partition \emph{admissible}. 
If $p(n)$ denotes the number of unrestricted partitions of $n$, it is well-known and easy
to prove that the number $P(n)$ of admissible partitions of $n$ is
equal to $p(n) - p(n-1)$.  The number of admissible partitions of $n$ increases
rather rapidly with $n$; e.g., $P(100) = 21{,}339{,}417$.  To get a
rough idea of the size of $P(n)$ for general $n$, we mention in passing that 
the asymptotic formula
\begin{equation} \label{Mein} P(n) \sim \frac{\pi}{12\cdot 2^{1/2} \ n^{3/2}} \ e^{\pi (2n/3)^{1/2}}
 \mbox{ as $n\to\infty$}
\end{equation} may be deduced from a theorem of~\citet{M}.
Equation~\eqref{Mein} is analogous to the famous asymptotic formula
of~\citet[p. 79, Eq. 1.41]{HR} for the unrestricted partition function
\[ p(n) \sim \frac{1}{4\cdot 3^{1/2} \ n} \ e^{\pi (2n/3)^{1/2}}  \mbox{ as $n\to\infty$}. \]
 
  Associated with each admissible 
partition $\lambda$ of $n$ is an estimator $\hat{\sigma} = \hat{\sigma}_\lambda$ 
defined in~\eqref{sigmahatdef};
for a given $n$, the admissible partition $\lambda$ that 
corresponds to the $\hat{\sigma}_\lambda$ of minimum 
variance will be called the \emph{optimal partition} of $n$.
  
~\citet{GW} performed extensive computations for $2\leq n \leq 100$, 
and showed that the optimal partition of $n$, when the underlying distribution is normal,
uses as many $8$'s as
possible (with occasional $7$'s or $9$'s to adjust for the fact that not every $n$ is a
multiple of $8$), except for sporadic exceptions that occur for small $n$.
~\citet{GW} did not supply  a rigorous proof that their assertions held for
$n>100$, perhaps 
 owing to the lack of closed form expressions for the expected value and variance of
 the range when the underlying distribution is normal.
 
   Here we investigate the analogous problem in the case where the underlying 
distribution is exponential.

\section{Precise Statement of the Problem}
Let \[ X_1, X_2, \dots, X_n \dist \Exp(\theta); \] i.e. each $X_i$ has pdf
\[ f(x; \theta) = \frac{1}{\theta} e^{-x/\theta}, \] for $x\geq 0$ and $0$ otherwise, 
where $\theta>0$.
Since $\Var(X_i) = \theta^2$, by estimating the population standard deviation $\sigma$, we
are equivalently estimating the parameter $\theta$.  

\begin{remark} \label{remark}
  For a given partition $\lambda=(n_1, n_2, \dots, n_m)$ of $n$ with length $m$,
  we understand that the $m$ subsamples 
  of the random variables $X_1, X_2, \dots, X_n$ are to be
  \begin{multline*} \{ X_1, X_2, \dots, X_{n_1} \}, \{ X_{n_1+1}, \dots, X_{n_1+n_2} \}, \\ \dots,
   \{ X_{n_1+n_2 + \cdots + n_{m-1}+1}, \dots, X_{n_1+n_2+ \cdots + n_m} \}  . \end{multline*}
\end{remark}

For each admissible partition of $n$, we have 
$f_1 = 0$ since each part must be at least $2$, and $f_{n-1} = 0$ because
a partition of $n$ with no $1$'s clearly cannot have $n-1$ as a part.

  Following the notation in~\cite{GW}, let
  \[ d_n = \frac{\E(R_n)}{\sigma} \]
  and
  \[ k_n^2 = \frac{\E(R_n - d_n \sigma)^2}{\sigma^2} = \frac{\Var(R_n)}{\sigma^2}. \]
  
  It is well-known (see, e.g., \citet[p. 52, Ex. 3.2.1]{DN}) that 
    \[ d_n = \theta^{-1} H_{n-1,1} \mbox{  and  } k_n^2 = \theta^{-2} H_{n-1,2} \]
 where \[ H_{n,j} = \sum_{i=1}^n \frac{1}{i^j} \] are generalized harmonic numbers.

   A linear combination of random variables which gives the 
 minimum variance unbiased estimate has
 coefficients which are inversely proportional to the variances of the variables.
 Consequently, let
 \begin{equation} \label{CoeffDef}
  a_i =  \left(  {\frac{d_{n_i}}{k^2_{n_i}}} \right)  
  \left( {\sum_{j=1}^m \frac{d_{n_j}^2}{k^2_{n_j}}} \right)^{-1}. \end{equation}
 
 \begin{proposition}
 The estimator $$\hat{\sigma} = \sum_{i=1}^m a_i R_{n_i},$$ 
 where $a_i$ is defined in~\eqref{CoeffDef}, is an unbiased estimator
 of $\sigma$.
  \end{proposition}
  \begin{proof}
  \[ \E(\hat{\sigma}) = \E\left( \sum_{i=1}^m a_i R_{n_i} \right) = 
  \sum_{i=1}^m a_i \E(R_{n_i}) =  \sum_{i=1}^m a_i d_{n_i} \sigma = \sigma. \]
  \end{proof}

Now we state the main theorem:

\begin{theorem}[The Exponential ``Rule of Fours''] \label{main}
   Fix an integer $n>1$.  Write $n = 4q + r$, where
   $r= r(n) =\Mod(n,4)$, the least nonnegative residue of $n$ modulo $4$. 
      The optimal partition
  of $n$
  is as follows:
  \begin{itemize}
      \item$ \langle 4^q \rangle$, if $r=0$;   
      \item $\langle 4^{q-r} 5^r \rangle$,    if $r=1$ or $2$ and $q\geq r$;
      \item $\langle 3^1 4^{q} \rangle $,     if $r=3$ and $q\geq 1$;                                                                                    
       \item $\langle n \rangle = (n) $, if $2\leq n \leq 5$; and
       \item $\langle 3^2 \rangle = (3,3)$, if $n=6$.
   \end{itemize}
\end{theorem}

\begin{example}
Suppose we have the sample size $n=22$ which implies $q=5$ and $r=2$.  Thus,
according to Theorem~\ref{main}, the optimal partition of $22$ is $(5,5,4,4,4)$
and this in turn gives
\begin{multline*}
\hat{\sigma} =  \frac{5880}{27133} \Big[  \rng(X_1, \dots, X_5)  + \rng(X_6, \dots, X_{10}) +
\Big]  \\ + \frac{8118}{27133} \Big[ \rng(X_{11}, ..., X_{14}) +
   \rng(X_{15} , \dots, X_{18} ) + \rng(X_{19},\dots, X_{22}) \Big],
\end{multline*}
where for $i<j$, $\rng(X_i, \dots, X_j)$ denotes the range of the subsample 
$X_i, X_{i+1}, \dots, X_j$ of the sample $X_1, X_2, \dots, X_n$.  
\end{example}

\section{Proof of Theorem~\ref{main}}
Next, Theorem~\ref{main} will be reformulated into an integer linear program.
Upon solving the equivalent integer linear program, we will have proved Theorem~\ref{main}.

\subsection{Reformulation as an integer linear program}

For any given positive integer $j$, let us define
\[ C_j= \frac{d_j^2}{k_j^2} = \frac{\big( E(R_j) \big)^2}{\Var(R_j)} = 
\frac{H_{j-1,1}^2}{H_{j-1,2}} ,\]
observing that $C_j$ is independent of $\sigma$. 
Thus, the $C_j$ are
known absolute constants.
For a given random sample of size $n$, we seek a partition $\lambda$ of $n$
such that
$\Var \hat{\sigma}$ is minimized when the sample is partitioned according to
$\lambda$.

Notice that
\begin{equation}\label{VarSigmaHat} \Var(\hat{\sigma}) =  \Var \left( \sum_{i=1}^m a_i R_{n_i} \right) 
= \sum_{i=1}^m a_i^2 \Var(R_{n_i})  
= \frac{\sigma^2}{\sum_{i=1}^m \left( \frac{d_{n_i}}{k_{n_i}} \right)^2 } = 
\frac{\sigma^2}{\sum_{i=1}^m C_{n_i}}
.\end{equation}
Thus we seek the partition of $n$ which causes the denominator in the
rightmost expression of~\eqref{VarSigmaHat} to be as large as possible.

 For any given $n$, an optimal partition is given by a solution of the following integer
 linear program:

{\begin{equation*}  
 \mbox{maximize} \left(  C_2 f_2 + C_3 f_3 + \cdots + C_n f_n  \right)  \end{equation*}
 \centerline{subject to}
 \begin{equation} 2f_2 + 3f_3 + 4f_4 + \cdots  + nf_n = n 
 \tag{EKP}\end{equation}
 \[ f_2, f_3, f_4, \dots, f_n \geq 0 \]
  \[ f_2, f_3, f_4, \dots, f_n \in \mathbb Z. \]}
  
  The objective function is the denominator in the far right member of~\eqref{VarSigmaHat}.
  The constraints guarantee that we search only over admissible partitions
of $n$. 
  
  The program (EKP) is an instance of the equality-constrained ``knapsack problem."  
While integer linear programming problems are in general NP-complete, we will be able
to exploit the special structure to provide an optimal solution to (EKP).

\subsection{Statement and proof of the key lemma}
But first, our solution is contingent upon the following lemma. 
\begin{lemma} \label{lem}
  The function $C_n / n$ attains its maximum value at $n=4$.
\end{lemma}
\begin{remark}  The significance of maximizing function $C_n/n$ is as follows: given a
sample of size $N$ that is partitioned into $N/n$ samples of size $n$ each, the denominator in the rightmost member of~\eqref{VarSigmaHat} is $N C_n/ n$.  The sample
size $N$ is fixed, but we wish to find the part size $n$ that maximizes $C_n/n$ in
order to maximize the denominator of the rightmost member of~\eqref{VarSigmaHat}.
Of course, this can only be done exactly in the case where the sample size $N$ is a 
multiple of the part size $n$, accounting for the fact that while $n=4$ gives the
optimal part size, for sample sizes that are not multiples of $4$, the optimal
partition includes some parts of size $3$ or $5$.
\end{remark}
\begin{proof}[Proof of Lemma~\ref{lem}]
  We need to show that for $n\neq 4$, 
     \[ \frac{C_n}{n} =
     \frac{H_{n-1,1}^2}{n H_{n-1,2}} < \frac{H_{3,1}^2}{4 H_{3,2}} = \frac{121}{196}. \]
  We will utilize the elementary inequalities 
  \begin{equation} \label{ineq1}
     H_{n,1} < 1 + \log n
  \end{equation}   and
  \begin{equation} \label{ineq2}
    H_{n,2} > 1 - \frac{1}{n+1} .
  \end{equation}
 
Inequality~\eqref{ineq1} may be deduced as follows: 
note that $\int_x^{x+1} 1/t \ dt = \log(x+1)-\log x$ for $x>0$.  Then observe that 
$1/x > 1/t > 1/(x+1)$ in the integrand, thus $1/x > \log(x+1) - \log(x) > 1/(x+1)$ and thus
$\log(n+1) < H_{n,1} <1 + \log n$ follows by summation.  Inequality~\eqref{ineq2}
follows from the fact that 
\[ H_{n,2} = \sum_{j=1}^n \frac{1}{j^2} > \int_1^{n+1} \frac{1}{x^2} \ dx = 1-\frac{1}{n+1}. \]
  
Applying~\eqref{ineq1} and~\eqref{ineq2}, we obtain   
\begin{align}
  \frac{C_n}{n} =
   \frac{H_{n-1,1}^2}{n H_{n-1,2}} < \frac{ (1+\log(n-1))^2 }{n\left(1 - \frac{1}{n} \right) }
   = \frac{ (1+\log(n-1))^2 }{n-1}.
\end{align}
Let $h(n) = { (1+\log(n-1))^2 }/(n-1)$.
Thus, $C_n/n$ is bounded above by $h(n)$.
Let us now temporarily consider $n$ to be a real variable with domain $n>1$.  
Notice that
\begin{equation} \label{dhdn}
\frac{dh}{dn} = \frac{d}{dn} \left\{  \frac{ (1+\log(n-1))^2 }{n-1}  \right\} =
 \frac{ [1+\log(n-1)] [1-\log(n-1) ]  }{(n-1)^2} .
\end{equation}
Since $1+\log(n-1)>0$ and $(n-1)^2>0$ for $n>1$, 
$dh/dn$ is negative whenever $\log(n-1)>1$, i.e. when $n > e+1$.  Thus $h(n)$ 
is decreasing for all integers $n\geq 4$.   Now $h(34) < 121/196$.   Thus 
$C_n/n < h(n) < 121/196$ for all $n\geq 34$.  That $C_n/n < 121/196$ for $n=2,3$ and
$5\leq n \leq 33$ can be verified by direct computation.
\end{proof}

\subsection{Solving (EKP) by the method of group relaxation}
With Lemma~\ref{lem} in hand, we now proceed to solve (EKP).
  Fix $n = 4q + r$, with $r=\Mod(n,4)$.
We now closely follow the treatment of~\citet[\S7.2, p. 184 ff]{JL} to solve (EKP).

In general, the first step is to seek an upper bound for 
\[  \max \sum_{j=2}^n C_j f_j, \] and then, if necessary, initiate a branch-and-bound procedure.
In the case of interest, this initial step will be sufficient to find an optimal
solution to (EKP).
By Lemma~\ref{lem} we have that
\[ \frac{C_4}{4} = \max \left\{ \frac{C_j}{j} : 2 \leq j \leq n \right\} . \]
Relax the nonnegativity restriction on $f_4$, and solve for $f_4$ in terms of the other 
$f_j$'s, namely 
\begin{equation} \label{f4}
f_4 = \frac n4 - \frac 14\underset{j\neq 4}{\sum_{2\leq j\leq n}} j f_j,
\end{equation} to obtain the ``group relaxation":

\begin{equation*} 
 \frac{n C_4}{4} + \max \underset{j\neq 4} { \sum_{2\leq j \leq n}}
 \left( C_j - \frac{j C_8}{8} \right) f_j \end{equation*}
 \centerline{subject to}
 \begin{equation}   \underset{j\neq 4} { \sum_{2\leq j \leq n}} j f_j = n - 4f_4; 
 \tag{GR}\end{equation}
 \[ f_2, f_3, f_5, f_6, f_7, \dots, f_n \geq 0 ;\]
  \[ f_2, f_3, f_4, \dots, f_n \in \mathbb Z. \]
The group relaxation (GR) is in turn equivalent to:
\begin{equation*}   
 \frac{n C_4}{4} - \min \underset{j\neq 4} { \sum_{2\leq j \leq n}}
 \left( -C_j + \frac{j C_4}{4} \right) f_j \end{equation*}
 \centerline{subject to}
 \begin{equation}   \underset{j\neq 4} { \sum_{2\leq j \leq n}} j f_j \equiv r \pmod{4};
 \tag{GR$'$}\end{equation}
 \[ f_2, f_3, f_5, f_6,  f_7, f_8, \dots, f_n \geq 0; \]
  \[ f_2, f_3, f_4, \dots, f_n \in \mathbb Z. \]

We note that every feasible solution in (GR$'$) corresponds to a value of $f_4$ in (GR),
via Equation~\eqref{f4}.  
If we are lucky, the optimal solution we find to (GR$'$) yields an 
$f_4\geq 0$, and then we will have also found an optimal solution to (EKP).
(If we are not lucky and the corresponding $f_4$  in (GR) is negative, then we
must embark on a branch-and-bound procedure with potentially many iterations.)
Here, however, we will show that for all $n > 6$, we are indeed lucky, and the
optimal solution to (EKP) will indeed be found directly with no need to initiate 
branch-and-bound.
 
   In order to solve (GR$'$), we form a weighted directed 
 multigraph $G$ as follows.  Let the vertex set 
be given by
$V(G)= \{ 0, 1, 2, 3 \}$.    For each $v\in V(G)$ and $j\in \{ 2,3, 5, 6, 7, 8,\dots, n \}$, 
there is an edge $e(v,j)$ of weight $-C_j + j C_4/4$ 
from $v$ to $\Mod( v+j, 4 )$.   Note well that in this notation,
 the ``ending vertex" of edge $e(v,j)$ is
 \emph{not} $j$
but rather $\Mod(v+j,4)$. We seek a minimum weight directed walk from $0$ to $r$.
Each time we include the edge $e(v,j)$ in the diwalk, we increment $f_j$ by $1$.
Since for each $j\neq 4$, the edge weight $-C_j + j C_4/4 \geq 0$, we may use the
algorithm of~\citet{EWD} to find a minimum weight dipath from $0$ to $r$.

We first dispense with the trivial case $r=0$. 
Here we immediately have the optimal solution (via the empty path from $0$ to
$0$)
\[  f_j =  \left\{ \begin{array}{ll} q &\mbox{if $j=4$}\\
                                               0  & \mbox{otherwise} 
                                         \end{array} \right. . \]

We now move on to the remaining cases $1\leq r \leq 3$.
Notice that connecting any two vertices in $G$ there are multiple directed edges.  
For each $j \in \{ 2, 3, 4, \dots, n\} \setminus \{ 4 \}$, 
connecting  vertex $v$ to vertex $\Mod(v+j,4)$, we have the following edges:
$e(v,j), e(v,2j), e(v,3j), \dots, e(v,mj) $, where $m$ is the largest integer such that
$v+mj\leq n$.  All of these edges have different weights, and since we seek a 
minimum weight diwalk, for each $j$, we may safely remove all but the one of lowest weight,
resulting in a much less ``cluttered" digraph $G'$.
 Thus $G'$ is a digraph with four vertices and $12$ directed edges
(as each of the $4$ vertices now has exactly one directed edge to each of the $3$ other vertices).

\begin{tabular}{| c| c | l |}
\hline
edge & connecting vertices  & weight \\
\hline
$e(0,1)$ & $0\rightarrow 1$ & $-C_5 + 5 C_4/4 = 305/8036 \approx 0.038$\\
$e(1,1)$ & $1\rightarrow 2$ & $-C_5 + 5 C_4/4 = 305/8036 \approx 0.038$\\ 
$e(2,1)$ & $2\rightarrow 3$ & $-C_5 + 5 C_4/4 = 305/8036 \approx 0.038$\\
$e(0,2)$ & $0\rightarrow 2$ & $-C_2 + 2 C_4/4 = 23/98 \approx 0.235$\\
$e(1,2)$ & $1\rightarrow 3$ & $-C_2 + 2 C_4/4 = 23/98 \approx 0.235$\\
$e(2,2)$ & $2\rightarrow 0$ & $-C_2 + 2 C_4/4 = 23/98 \approx 0.235$\\
$e(0,3)$ & $0\rightarrow 3$ & $-C_3 + 3 C_4/4 = 51/980 \approx 0.052$\\
$e(1,3)$ & $1\rightarrow 0$ & $-C_3 + 3 C_4/4 = 51/980 \approx 0.052$\\
$e(2,3)$ & $2\rightarrow 1$ & $-C_3 + 3 C_4/4 = 51/980 \approx 0.052$\\
\hline
\end{tabular}

\vskip 3mm
Perform the algorithm of~\citet{EWD} on the weighted directed graph $G'$ to find that
the minimum weight directed paths are as follows:
\vskip 3mm
\begin{tabular}{|c|c|c|c|}
\hline
from $\rightarrow$ to & using edges  & total weight & resulting nonzero values  \\
\hline
$0\to 1$ & $e(0,1)$  & $305/8036$ & $f_5 = 1, f_4 = q-1$ \\
$0\to 2$ & $e(0,1)$ and $e(1,1)$ & $610/8036$ & $f_5 =2, f_4 = q-2$ \\
$0\to 3$ & $e(0,3)$     &$51/980$ & $f_3 = 1, f_4 = q$
\\
\hline
\end{tabular}
\vskip 3mm
Thus we arrive at precisely the desired result.

Since the value of $f_4$ obtained from the group relaxation (GR) (see
Eq.~\eqref{f4}) is always nonnegative,
then group relaxation solves the original equality knapsack problem (EKP), and thus
no branch-and-bound procedure need be undertaken.  

With the solution of (EKP), we
have proved Theorem~\ref{main} for $n>6$.  For $2\leq n \leq 6$, the theorem can
be established by direct calculation. \qed

\section{Conclusion}
We have proved that an unbiased estimator from an exponential population formed by taking
an appropriately weighted sum of the ranges of subsamples obtained by partitioning the
original sample of size $n$ has minimal variance when the subsamples are each of size $4$
(or as close as possible when $n$ is not a multiple of $4$).  This contrasts with the
work of~\citet{GW}, where the optimal subsample sizes are $8$, when the population is
normal.  A similar analysis could be applied to estimate the standard deviation from 
populations with other distributions.  For example, based on some
preliminary calculations, in the case of the Rayleigh distribution,
the optimal subsample sizes appear to be $8$, and in the case of the $\chi^2$ distribution,
the optimal subsample sizes appear to be $3$.  Another variant perhaps worth investigating
is seeking the optimal partition of the sample where instead of the range of subsamples,
we consider various quasi-ranges.

\citet[pp. 9--11]{HB} compare the efficiency of estimators of $\sigma$ based on quasi-ranges
with the Grubbs--Weaver estimator for $\sigma$ in the case of a normal parent population and
remark [p. 11] that the asymptotic efficiency of the Grubbs--Weaver estimator is
75.38 percent.  In future work we would like to obtain analogous results for various
non-normal parent distributions, including the exponential distribution that was 
studied in this present paper.

\section*{Acknowledgments}
The authors thank the Editor and the anonymous referee for helpful suggestions that
improved the paper.

\bibliographystyle{natbib-harv}

\begin{thebibliography}{}

\bibitem[David and Nagaraja, 2003]{DN}
David, H. A. and Nagaraja, H. N., 2003.
\emph{Order Statistics}, 3rd ed.,  Wiley Interscience.  Wiley Series in Probability and Statistics.


\bibitem[Dijkstra, 1959]{EWD}
Dijkstra, E. W., 1959,
A note on two problems in connexion with graphs,
\emph{Numerische Mathematik} 1, 269--271.

\bibitem[Grubbs and Weaver, 1947]{GW}
Grubbs, F. E. and Weaver, C. L., 1947,
 The best unbiased estimate of population
standard deviation based on group ranges, \emph{J. Amer. Stat. Assoc.} 42, 224--241.

\bibitem[Hardy and Ramanujan, 1918]{HR}
Hardy, G. H. and Ramanujan, S., 1918,
Asymptotic formul\ae\ in combinatory analysis,
\emph{Proc. London Math. Soc.} 17, 75--15.

\bibitem[Harter and Balakrishnan, 1996]{HB}
Harger, H. Leon and Balakrishnan, N., 1996,
\emph{CRC Handbook of Tables for the Use of Order Statistics in Estimation},
CRC Press.

\bibitem[Lee, 2004]{JL}
Lee, J., 2004. \emph{A First Course in Combinatorial Optimization},
Cambridge Texts in Applied Mathematics, Book 36.  Cambridge University Press.

\bibitem[Meinardus, 1954]{M}
Meinardus, G., 1954,
Asymptotische Aussagen \"uber Partitionen,
\emph{Math. Z.} 61, 289--302.

\end{thebibliography}

\end{document}